\documentclass[12pt]{amsart}
\usepackage{amssymb,amsmath,amsthm,epsfig}

\usepackage{amsmath}
\usepackage{amsthm}
\usepackage{amssymb}
\usepackage{amsfonts}
\usepackage{latexsym}
\usepackage{verbatim}
\usepackage{graphicx}

\newcommand{\bmat}[1]{\left[\begin{array}{#1}}
\newcommand{\bbmat}[0]{\end{array}\right]}

\newcommand{\pmat}[0]{\begin{pmatrix}}
\newcommand{\ppmat}[0]{\end{pmatrix}}

\newcommand{\R}[0]{\mathbb{R}}

\newcommand{\N}[0]{\mathbb{N}}
\newcommand{\C}[0]{\mathbb{C}}

\newcommand{\vp}{\varepsilon}

\newtheorem{thm}{Theorem}
\newtheorem{fact}[thm]{Fact}

\newtheorem{lem}[thm]{Lemma}
\theoremstyle{definition}
\newtheorem{defin}[thm]{Definition}

\allowdisplaybreaks

    \title{Moving finite unit tight frames for $S^n$}
    \date{\today}
    \author{Daniel Freeman}
\address{Department of Mathematics and Computer Science\\
St Louis University\\
St Louis, MO 63103  USA} \email{dfreema7@slu.edu}

\author{Ryan Hotovy}
\address{Department of Mathematics, Texas A\&M University\\
College Station, TX 77843, USA} \email{rhotovy@math.tamu.edu}

\author{Eileen Martin}
\address{Department of Mathematics, University of Texas at Austin, 1 University Station
C1200, Austin TX 78712-0257} \email{eileen.r.martin@utexas.edu}

\thanks{
The research of all authors was supported by the ``Wavelets and
Matrix Analysis'' REU at Texas A\&M which was supported by the
National Science Foundation.  The first author was supported by a
grant from the National Science Foundation.}


\begin{document}

\begin{abstract}
Frames for $\R^n$ can be thought of as redundant or linearly
dependent coordinate systems, and have important applications in
such areas as signal processing, data compression, and sampling
theory.  The word ``frame'' has a different meaning in the context
of differential geometry and topology. A moving frame for the
tangent bundle of a smooth manifold is a basis for the tangent space
at each point which varies smoothly over the manifold. It is well
known that the only spheres with a moving basis for their tangent
bundle are $S^1$, $S^3$, and $S^7$. On the other hand, after
combining the two separate meanings of the word ``frame'',  we show
that the $n$-dimensional sphere, $S^n$, has a moving finite unit
tight frame for its tangent bundle if and only if $n$ is odd.  We
give a procedure for creating vector fields on $S^{2n-1}$ for all
$n\in\N$, and we characterize exactly when sets of such vector
fields form a moving finite unit tight frame.
\end{abstract}\maketitle

\section{Introduction}\label{S:1}
When analyzing and processing a signal, such as a sound wave or
digital picture, it is often useful to represent the signal in terms
of a basis.
 For instance, a sound wave can be represented
in terms of its harmonic frequencies by a fourier series. With this
representation, the sound wave can be processed to remove static by
setting to 0 the high frequency coefficients, or be compressed for
smaller storage on a computer by discretizing the coefficients and
setting small coefficients to 0. Representing a function with
respect to a basis is also useful in transmitting signals, as all
that is needed to be sent are the basis coefficients instead of
every value of the function. There is a problem with this
transmission procedure though when error is introduced into the
system.  If one of the basis coefficients is lost or corrupted
during transmission, then an entire dimension is lost and cannot be
recovered.  This is where using a frame instead of a basis can be
useful.  A frame is essentially a redundant coordinate system in
that a frame uses more vectors than necessary when representing a
signal.  However, this redundancy helps to mitigate the error due to
the loss of coordinates through transmission by effectively
spreading the loss over the whole space instead of localizing it to
certain dimensions. Given $n\in\N$, a {\em tight frame for $\R^n$}
is a sequence $(f_i)_{i=1}^k\subset \R^n$ with $k\geq n$ such that
there exists a constant $C>0$ satisfying
\begin{equation}\label{RecForm}
x=\frac{1}{C}\sum_{i=1}^k \langle f_i,x\rangle f_i\quad\textrm{ for
all }x\in\R^n,
\end{equation}
where, $\langle\cdot,\cdot\rangle$ is the standard inner product on
$\R^n$ (the dot product). Thus, a frame allows for the exact
reconstruction of a vector from the sequence of frame coefficients
$\left(\langle f_i,x\rangle\right)_{i=1}^k$.
 We call a tight frame $(f_i)_{i=1}^k$ a {\em finite unit tight
frame, or FUNTF}, if $\|f_i\|=1$ for all $1\leq i\leq k$.  Note that
in the case $k=n$, a sequence $(f_i)_{i=1}^n\subset\R^n$ is a FUNTF
if and only if it is an ortho-normal basis. FUNTFs are the most
useful frames for signal processing, as they minimize mean squared
error due to noise \cite{GKK}.  This has motivated the study of
FUNTFs and in particular has led to interest in finding ways to
construct FUNTFs. In \cite{BF},  using a potential energy concept
inspired from physics, it is proven that there exists a finite unit
tight frame of $k$ vectors for $\R^n$ for all natural numbers $k\geq
n$. An explicit construction for FUNTFs is given in \cite{DFKLOW},
and in \cite{CFHWZ} the method of spectral tetris is introduced to
explicitly construct FUNTFs as well other types of frames.

In the context of differential geometry and topology, the word
``frame'' has a different meaning.  A moving frame for the tangent
bundle of a smooth manifold can be thought of as a basis for the
tangent space at each point on the manifold which moves continuously
over the manifold.  To avoid confusion, we will refer to moving
frames in this context as moving bases. If $n\in\N$ and $M^n$ is an
$n$-dimensional smooth manifold, then a {\em vector field} is a
continuous function $f:M^n\rightarrow T M^n$ such that $f(p)\in T_p
M^n$ for all $p\in M$.  That is, a vector field continuously assigns
to each point $p\in M^n$, a vector $f(p)$ in the tangent space $T_p
M$. Thus, a moving basis for the tangent bundle of a $n$-dimensional
smooth manifold $M^n$ is a sequence of vector fields $(f_i)_{i=1}^n$
such that $(f_i(p))_{i=1}^n$ is a basis for $T_p M$ for all $p\in
M^n$.

Of particular interest in the theory of vector fields and moving
bases is the case where the underlying manifold $M^n$ is the
$n$-dimensional sphere $S^n$. Note that $S^1$ can be represented by
the unit circle in $\R^2$ and that $S^2$ can be represented by the
unit sphere in $\R^3$. The famous Hairy Ball Theorem states that if
$n$ is even then there does not exist a nowhere zero vector field
for $S^n$.  As a moving basis for $S^n$ would consist of $n$ nowhere
zero vector fields, the manifold $S^n$ cannot have a moving basis
when $n$ is even.  On the other hand, it can be simply proven that
$S^1$, $S^3$, and $S^7$ all have a moving basis for their tangent
space. To give a brief sketch of this proof, we identify $S^1$ with
the unit circle of the complex numbers $\C$, $S^3$ with the unit
sphere of the quaternions $\mathbb{H}$, and $S^7$ with the unit
sphere of the octonions $\mathbb{O}$. Each of these division
algebras contain the real numbers, so in particular we have that
$1\in S^1\subset\C$, $1\in S^3\subset\mathbb{H}$, and $1\in
S^7\in\mathbb{O}$. Now for each of the cases $n=1,3,7$, we can pick
a basis $(v_i)_{i=1}^n$ for the tangent space $T_1 S^n$. To move
this basis $(v_i)_{i=1}^n\subset T_1 S^n$ from $1\in S^n$ to a
different point $a\in S^n$, we simply multiply each vector in the
basis by $a$ giving $(a v_i)_{i=1}^n\subset T_a S^n$. This gives a
moving basis for $S^n$ for the cases $n=1,3,7$. This proof will not
work for other values of $n$ as $\C$, $\mathbb{H}$, and $\mathbb{O}$
are the only finite dimensional division algebras over $\R$ besides
$\R$ itself.  The question of does $S^n$ have a moving basis for its
tangent bundle for any other values of $n$ was an important open
problem in differential topology and was solved negatively in 1958
by Michel Kervaire and independently by Raoul Bott and John Milnor.

$S^n$ has a basis for its tangent space at each point which moves
smoothly over $S^n$ if and only if $n\in\{1,3,7\}$.  This motivates
us to question what happens when we weaken the condition of basis to
that of finite unit tight frame.  By combining the two definitions
of the word ``frame'', we are led to studying tight frames which
vary smoothly over a manifold.  As the reconstruction formula
(\ref{RecForm}) uses an inner product, we need to consider
Riemannian manifolds, where the Riemannian metric gives a smoothly
varying inner product for the tangent space at each point on the
manifold.  Note that any smooth submanifold of $\R^n$ is naturally a
Riemannian manifold where the Riemannian metric is given by the
standard inner product on $\R^n$.
\begin{defin}\label{frame}
Let $M^n$ be an $n$-dimensional Riemannian manifold with Riemannian
metric $\langle\cdot,\cdot\rangle_a$ for each $a\in M^n$. Let $k\geq
n$, and $f_i:M\rightarrow T M$ be a vector field for all $1\leq
i\leq k$. We say that $(f_i)_{i=1}^k$ is a {\em moving tight frame}
for the tangent bundle of $M^n$ if $(f_i(a))_{i=1}^k$ is a tight
frame for $T_a M^n$ for all $a\in M^n$.  That is, for all $a\in
M^n$, there exists $C>0$ such that
$$x=\frac{1}{C}\sum_{i=1}^k \langle x,f_i(a)\rangle_a f_i(a)\qquad\textrm{ for all
}x\in\pi^{-1}(a).
$$
We say that $(f_i)_{i=1}^k$ is a {\em moving finite unit tight frame
(FUNTF)} for the tangent bundle of $M^n$ if it is a moving tight
frame and $\|f_i(a)\|=\sqrt{\langle f_i(a),f_i(a)\rangle}=1$ for all
$1\leq i\leq k$ and $a\in M^n$.
\end{defin}

If $n\in\N$ is even, then the sphere $S^n$ does not have a nowhere
zero vector field.  As a moving FUNTF consists of unit vector
fields, no even dimensional sphere can have a moving FUNTF for its
tangent bundle. On the other hand, we will show that every odd
dimensional sphere has a moving FUNTF for its tangent bundle.

The concept of a moving tight frame was first applied in 2009 when
P. Kuchment proved that particular vector bundles over the torus,
which arise in mathematical physics, have natural moving tight
frames but do not have moving bases \cite{K}.  In \cite{FPWW}, it
was shown that the dilation theorem for tight frames in $\R^n$ can
be extended to moving tight frames.  The relationship between frames
for Hilbert spaces and manifolds was also considered in a different
context by Dykema and Strawn, who studied the manifold structure of
collections of FUNTFs under certain equivalent classes \cite{DySt}.

For terminology and background on vector bundles and smooth
manifolds see \cite{L}. For terminology and background on frames for
Hilbert spaces see \cite{C} and \cite{HKLW}.

Most of the research contained in this paper was conducted at the
2010 Research Experience for Undergraduates in Matrix Analysis and
Wavelets organized by Dr David Larson.  The first author was a
research mentor for the program, and the second and third authors
were participants.  We sincerely thank Dr Larson for his advice and
encouragement.

\section{Tight frames}\label{S:1}

Before discussing moving tight frames, we give some basic results
about how to check whether or not a sequence of vectors is a tight
frame for a Hilbert space.  For $n\in\N$, we denote the unit vector
basis for $\R^n$ by $(e_i)_{i=1}^n$.  As the reconstruction formula
(\ref{RecForm}) is linear, it holds for all $x\in\R^n$ if and only
if it holds for the vectors in the orthonormal basis
$(e_i)_{i=1}^n$. Furthermore, if $1\leq p\leq n$ and $x\in\R^n$,
then $x=e_p$ if and only if $\langle x, e_q\rangle=\delta_{p,q}$ for
all $1\leq q\leq n$.  This leads us to the following fact.

\begin{fact}\label{factBasis}
Let $(f_i)_{i=1}^k\subset\R^n$.  The sequence $(f_i)_{i=1}^k$ is a
tight frame for $\R^n$ if and only if there exists a constant $C>0$
such that
$$\frac{1}{C}\sum_{i=1}^k\langle f_i, e_p\rangle\langle
f_i,e_q\rangle=\delta_{p,q}\quad\quad\textrm{ for all }1\leq p,q\leq
n.
$$
\end{fact}

Given $n\in\N$ and $a\in S^{2n-1}\subset\R^{2n}$, we will be
interested in identifying when a sequence $\{f_i(a)\}_{i=1}^k$ is a
tight frame for the subspace $T_a S^{2n-1}\subset \R^{2n}$. However,
Fact \ref{factBasis} may only be used to check if a sequence of
vectors is a tight frame for the entire space, not just a subspace.
 We will use the following fact to get around this problem.

\begin{fact}\label{FactUnion}
Let $n\in\N$ and $X,Y\subset \R^n$ such that $X=Y^\bot$.  If
$(x_i)_{i=1}^k\subset X$ is a tight frame for $X$ with frame
constant $C>0$ and $(y_i)_{i=1}^\ell\subset Y$, then
$(y_i)_{i=1}^\ell$ is a tight frame for $Y$ with frame constant $C$
if and only if $(x_1,...,x_k,y_1,...,y_\ell)$ is a tight frame for
$\R^n$ with frame constant $C$.
\end{fact}

To apply Fact \ref{FactUnion}, we need to know the frame constant
for a given finite unit tight frame.  This is given by the following
fact, which can be deduced from Fact \ref{factBasis} by summing over
$1\leq p=q\leq n$.

\begin{fact}\label{FactCoef}
Let $n,k\in\N$ such that $k\geq n$. If $(f_i)_{i=1}^k\subset \R^n$
is a finite unit tight frame for $\R^n$ then $\frac{k}{n}$ is the
frame constant for $(f_i)_{i=1}^k$.
\end{fact}

\section{A moving FUNTF for $S^{2n-1}$}

We identify the sphere $S^n$ with the unit vectors in $R^{n+1}$.
That is, $S^n=\{(a_1,...,a_{n+1})\in\R^{n+1}:\sum_{i=1}^{n+1}
a_i^2=1\}$. With this representation, it is very easy to identify
the tangent space $T_a S^n$ for each $a\in S^n$.  We first note that
if $a\in S^n$, then the vector $a$ is normal to the tangent space
$T_a S^n$.  Thus, as $S^n$ is a manifold of 1 less dimension than
$R^{n+1}$, if $a\in S^{n}$, then a vector $b\in\R^{n+1}$ is in the
tangent space $T_a S^n$ if and only if $b$ is orthogonal to $a$.
That is, if $a=(a_1,...,a_{n+1})\in S^{n}$ and
$b=(b_1,...,b_{n+1})\in\R^{n+1}$ then $b\in T_a S^n$ if and only if
$\langle a, b\rangle=\sum_{i=1}^{n+1}a_i b_i=0$.  For example, in
the case of the circle, if $(a_1,a_2)\in S^1\subset\R^2$ then
$(-a_2, a_1)$ is the unit vector in the tangent line
$T_{(a_1,a_2)}S^1$ which points counter-clockwise.

\begin{figure}
\includegraphics[width=3in, trim=1.7in 6.6in 1.7in 0.8in, clip=true]{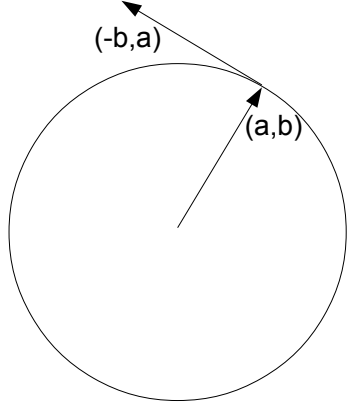}
\label{circle}
\caption{A unit-length tangent vector can be created at $(a,b) \in S^1$ by permuting the coordinates and multiplying one coordinate by -1 to give the vector $(-b,a)$.}
\end{figure}

Thus,  we can create a unit vector field for $S^1\subset \R^2$ by
switching the two coordinates and including a negative sign in the
first coordinate.  This method can be used to create vector fields
for higher dimensional spheres as well.  Given any $n\in\N$, we
create a unit vector field for $S^{2n-1}\subset\R^{2n}$ by switching
pairs of coordinates and inserting a negative sign into one
coordinate of each pair. For example, the map
$(a_1,a_2,a_3,a_4)\mapsto (a_2,-a_1,a_4,-a_3)$ is a unit vector
field on $S^{3}$.  To formalize this, if $(\vp_i)_{i=1}^{2n}\in
\{-1,1\}^{2n}$ and $(k_i)_{i=1}^{2n}\in\{1,2,...,2n\}^{2n}$ then we
define $U_{(\vp_i,k_i)_{i=1}^{2n}}:\R^{2n}\rightarrow\R^{2n}$ to be
the operator defined by
$U_{(\vp_i,k_i)_{i=1}^{2n}}(a_1,...,a_{2n})=(\vp_{k_1}
a_{k_1},...,\vp_{k_{2n}} a_{k_{2n}})$. As we are only interested in
$U_{(\vp_i,k_i)_{i=1}^{2n}}(a_1,...,a_{2n})$ when it switches pairs
of coordinates and multiplies one of the coordinates in each pair by
$-1$, we restrict ourselves to the set,
$$\mathbf{A}_{2n}:=\left\{U_{(\vp_i,k_i)_{i=1}^{2n}}:\R^{2n}\rightarrow\R^{2n}:\vp_i=-\vp_{k_i},\,
i=k_{k_i}\textrm{ for all }1\leq i\leq 2n\right\}.
$$
Note that the condition that $\vp_i=-\vp_{k_i}$ guarantees in
particular that $i\neq k_i$ for all $1\leq i\leq 2n$.  Thus, the set
$\mathbf{A}_{2n}$ is exactly the collection of operators which
switches pairs of coordinates and multiplies one of the coordinates
in each pair by $-1$.  The next lemma shows that the operators in
$\mathbf{A}_{2n}$ can be used to construct tangent vectors.
\begin{lem}\label{LemmaTangentSphere}
Let $n\in\N$, $U_{(\vp_i,k_i)_{i=1}^{2n}}\in \mathbf{A}_{2n}$, and
$a=(a_1,...,a_{2n})\in S^{2n-1}\subset\R^{2n}$.  Then,
$U_{(\vp_i,k_i)_{i=1}^{2n}}(a)\in T_a S^{2n-1}$.
\end{lem}
\begin{proof}
Recall that $U_{(\vp_i,k_i)_{i=1}^{2n}}(a)\in T_a S^{2n-1}$ if and
only if $\langle U_{(\vp_i,k_i)_{i=1}^{2n}}(a), a\rangle=0$.  We
have that,
$$2\langle U_{(\vp_i,k_i)_{i=1}^{2n}}(a),
a\rangle=2\sum_{i=1}^{2n}\vp_{k_i} a_{k_i} a_i=\sum_{i=1}^{2n}\vp_{k_i} a_{k_i} a_i +\vp_{k_{k_i}} a_{k_{k_i}} a_{k_i}=\sum_{i=1}^{2n}\vp_{k_i} a_{k_i} a_i +\vp_{i} a_{i} a_{k_i}=0.\\
$$
Thus, $\langle U_{(\vp_i,k_i)_{i=1}^{2n}}(a), a\rangle=0$ and hence
  $U_{(\vp_i,k_i)_{i=1}^{2n}}(a)\in T_a S^{2n-1}$.  The vector
$U_{(\vp_i,k_i)_{i=1}^{2n}}(a)$ is formed by permuting the
coordinates of $a$ and multiplying half the coordinates by $-1$.
None of these operations changes the norm, and thus
$\|U_{(\vp_i,k_i)_{i=1}^{2n}}(a)\|=1$ as $\|a\|=1$.
\end{proof}
 By
 Lemma~\ref{LemmaTangentSphere}, if $U\in \mathbf{A}_{2n}$, then we may
define a unit vector field $f_{U}:S^{2n-1}\rightarrow T S^{2n-1}$
 by $f_{{U}}(a):=
U(a)$ for all $a\in S^{2n-1}$.  In the introduction we used division
algebras over $\R$ to prove that $S^1$, $S^3$, and $S^7$ each have a
moving basis for their tangent bundle. However, for $n\in\{1,3,7\}$
it is also possible to choose $A\subset A_{2n}$ such that
$\{f_U\}_{U\in A}$ is a moving orthonormal basis for $S^n$. Indeed,
if $(a_1,a_2)\in S^1$ then $\{(-a_2,a_1)\}$ is an orthonormal basis
for $T_{(a_1,a_2)}S^1$, and if $(a_1,a_2,a_3,a_4)\in S^3$ then
$\{(-a_2,a_1,a_4,-a_3),(-a_3,-a_4,a_1,a_2),(-a_4,a_3,-a_2,a_1)\}$ is
an orthonormal basis for $T_{(a_1,a_2,a_3,a_4)}S^3$. Constructing a
moving orthonormal basis for $T S^7$ can be done similarly.  As no
other sphere has a moving basis for its tangent bundle, we now focus
on moving FUNTFs. A FUNTF is able to reconstruct every vector in
$\R^n$ exactly, and hence the vectors that comprise a FUNTF are in
this respect, distributed uniformly across $\R_n$.  Thus, it is
natural to assume that wether or not a set of vector fields of the
form $\{f_U\}_{U\in A}$ is a moving FUNTF depends on wether or not
the set $A$ in some respect uniformly switches pairs of coordinates
and in some respect uniformly inserts negative signs. With this in
mind, we introduce the following definition.
\begin{defin}
Let $n\in\N$ and $A\subset \mathbf{A}_{2n}$.  We say that $A$ is
{\em balanced} if the following two conditions are satisfied for all
$1\leq p,q\leq 2n$ with $p\neq q$.
\begin{itemize}
\item[i)] $\# \left\{U_{(\vp_i,k_i)_{i=1}^{2n}}\in A\,|\,
k_p=q\right\}=\frac{\# A}{2n-1}$,\\
and for all $1\leq r,s\leq 2n$ with $p,q,r,\textrm{ and } s$ all
being distinct integers,
\item[ii)] $\# \left\{U_{(\vp_i,k_i)_{i=1}^{2n}}\in A\,|\,
\vp_p \vp_q=-1,\,
\{k_r,k_s\}=\{p,q\}\right\}=\#\left\{U_{(\vp_i,k_i)_{i=1}^{2n}}\in
A\,|\,
\vp_p \vp_q=1,\,  \{k_r,k_s\}=\{p,q\}\right\}$.\\
\end{itemize}
\end{defin}
 For the sake
of convenience, if $A\subset\mathbf{A}_{2n}$ and $1\leq p,q\leq 2n $
are distinct integers then we denote $A_{p,q}=
\left\{U_{(\vp_i,k_i)_{i=1}^{2n}}\in A\,|\, k_p=q\right\}$, and if
$\vp\in\{-1,1\}$ and $1\leq p,q,r,s\leq 2n$ are distinct integers
then we denote $A_{p,q,r,s,\vp}=\left\{U_{(\vp_i,k_i)_{i=1}^{2n}}\in
A\,|\, \vp_p \vp_q=\vp,\, \{k_r,k_s\}=\{p,q\}\right\}$.

 Verifying that a
particular set $A\subset \mathbf{A}_{2n}$ is balanced requires
checking many different equations, and it is not immediately obvious
that balanced sets always exist. However, the following lemma shows
that the entire set $\mathbf{A}_{2n}$ itself is balanced.

\begin{lem}\label{LemmaBalanced}
For all $n\in\N$, the set $\mathbf{A}_{2n}$ is balanced.
\end{lem}

\begin{proof}
Let $n\in\N$.  We first count the number of ways to choose an
operator $U_{(\vp_i,k_i)_{i=1}^{2n}}\in\mathbf{A}_{2n}$. We may
count $(2n-1)$ ways to choose $k_1$ from
$\{1,...,2n\}\setminus\{1\}$. If $i_2$ is the first element of
$\{1,...,2n\}\setminus \{1,k_1\}$ then there are $(2n-3)$ choices
for $k_{i_2}$ from $\{1,...,2n\}\setminus\{1,k_1,i_2\}$.  Continuing
in this manner gives $(2n-1)(2n-3)...3\cdot1$ ways to choose
$(k_i)_{i=1}^{2n}$ for an operator
$U_{(\vp_i,k_i)_{i=1}^{2n}}\in\mathbf{A}_{2n}$. For each $1\leq
i\leq 2n$, there are two ways to choose $(\vp_i, \vp_{k_i})$. As
there are $n$ of these pairs, we have that,
$\#\mathbf{A}_{2n}=(2n-1)(2n-2)...3\cdot1\cdot2^n=\frac{(2n-1)(2n-2)...3\cdot1\cdot2^n
n!}{n!}=\frac{(2n)!}{n!}$.

Let $1\leq p,q\leq 2n$ such that $p\neq q$.  We now count the
elements in the set $A_{p,q}$. Given a choice of $(\vp_q,\vp_p)$ and
setting $k_p=q$ and $k_q=p$, there are $\# \mathbf{A}_{2n-2}$
choices for $(\vp_i,k_i)_{1\leq i\leq 2n, i\neq p,q}$ such that
$U_{(\vp_i,k_i)_{i=1}^{2n}}\in\mathbf{A}_{2n}$.  As there are 2
choices for $(\vp_p,\vp_q)$, we have that
$$\# A_{p,q}=2\#
\mathbf{A}_{2n-2}=2\frac{(2n-2)!}{(n-1)!}=\frac{2n(2n-1)(2n-2)!}{n(2n-1)(n-1)!}=\frac{(2n)!}{(2n-1)n!}=\frac{\#\mathbf{A}_{2n}}{2n-1}.
$$
Thus condition $i)$ is satified.  We now show that condition $ii)$
is satisfied.
 Let $1\leq p,q,r,s\leq 2n$ such that $p,q,r,\textrm{ and }s$ are all distinct integers.   For all $1\leq i\leq 2n$ and
$U_{(\vp_i,k_i)_{i=1}^{2n}}\in A_{p,q,r,s,-1}$, we let $\delta_i=-1$
if $i\in\{p,k_p\}$ and $\delta_i=1$ otherwise. We define a map
$\phi:A_{p,q,r,s,-1}\rightarrow A_{p,q,r,s,1}$ by
$\phi(U_{(\vp_i,k_i)_{i=1}^{2n}})=U_{(\delta_i\vp_i,k_i)_{i=1}^{2n}}$.
It is simple to check that this is a bijection, and hence $\#
A_{p,q,r,s,-1}=\# A_{p,q,r,s,1}$.
\end{proof}

Lemma \ref{LemmaBalanced} gives in particular that for all $n\in\N$
there exists a balanced subset of $\mathbf{A}_{2n}$.  The following
theorem then proves that there exists a moving FUNTF for $S^{2n-1}$
for every $n\in\N$.

\begin{thm}\label{ThmMain}
For all $n\in\N$, a subset $A\subset \mathbf{A}_{2n}$ is balanced if
and only if $\{f_U\}_{U\in A}$ is a moving FUNTF for $S^{2n-1}$.
\end{thm}
\begin{proof}
Let $n\in\N$ and $A\subset \mathbf{A}_{2n}$. Note that for all $a\in
S^{2n-1}$, we have that $a^\perp=T_a S^{2n-1}$. By Facts
\ref{FactUnion} and \ref{FactCoef}, for each $a\in S^{2n-1}$,
$\{f_U(a)\}_{U\in A}$ is a FUNTF for $T_a S^{2n-1}$ if and only if
$\left\{\sqrt{\frac{\#
A}{(2n-1)}}a\right\}\cup\left\{f_U(a)\right\}_{U\in A}$ is a tight
frame for $\R^{2n}$ with frame bound $\frac{\#{A}}{2n-1}$.

Assume that $A\subset \mathbf{A}_{2n}$ is not balanced.  We will
show that $\left\{f_U(a)\right\}_{U\in A}$ is not a FUNTF for $T_a
S^{2n-1}$ by proving that  $\left\{\sqrt{\frac{\#
A}{(2n-1)}}a\right\}\cup\left\{f_U(a)\right\}_{U\in A}$ is not a
tight frame for $\R^{2n}$.  As $A$ is not balanced, it fails either
condition $i)$ or condition $ii)$. We first assume that $A$ fails
condition $i)$.
 We have
that there exists $1\leq p,q\leq 2n$ such that $\#
A_{p,q}\neq\frac{\# A}{2n-1}$.  We will use Fact \ref{factBasis}
 to show that $\{\sqrt{\frac{\#
A}{(2n-1)}}\frac{1}{\sqrt{2}}(e_p+e_q)\}\cup\{U(\frac{1}{\sqrt{2}}(e_p+e_q))\}_{U\in
A}$ is not a tight frame for
$T_{\frac{1}{\sqrt{2}}(e_p+e_q)}S^{2n-1}$ by showing that the
reconstruction formula applied to $e_p$ is not orthogonal to $e_q$.
We let $a=\frac{1}{\sqrt{2}}(e_p+e_q)$.
\begin{align*} \frac{\#
A}{(2n-1)}\langle a, e_p\rangle\langle a, e_q\rangle+  \sum_{U\in A}
\langle U a,e_p\rangle  \langle U a, e_q\rangle
 &= \frac{\#A}{(2n-1)}\frac{1}{2}+ \sum_{U\in A_{p,q}}\vp_{k_q}\vp_{k_p}\frac{1}{\sqrt{2}}\frac{1}{\sqrt{2}}\\
 &= \frac{\#A}{(2n-1)}\frac{1}{2}+ \sum_{U\in
 A_{p,q}}-\frac{1}{2}\quad\quad\textrm{ as }k_p=q\\
 &=\frac{1}{2}\left(\frac{\#A}{(2n-1)}-
 \#A_{p,q}\right)\neq0
\end{align*}
Thus, $\left\{\sqrt{\frac{\#
A}{(2n-1)}}\frac{1}{\sqrt{2}}(e_p+e_q)\right\}\cup\left\{U(\frac{1}{\sqrt{2}}(e_p+e_q))\right\}_{U\in
A}$ is not a tight frame for $\R^{2n}$ by Fact \ref{factBasis} and
hence $(f_U)_{U\in A}$ is not a moving FUNTF for $S^{2n-1}$.  We now
assume that $A$ fails condition $ii)$.  There exist distinct
integers $1\leq p,q,r,s\leq 2n$  such that $\# A_{p,q,r,s,-1}\neq\#
A_{p,q,r,s,1}$.  We will show that $\left\{\sqrt{\frac{\#
A}{(2n-1)}}\frac{1}{\sqrt{2}}(e_p+e_q)\right\}\cup\left\{U(\frac{1}{\sqrt{2}}(e_p+e_q))\right\}_{U\in
A}$ is not a tight frame for
$T_{\frac{1}{\sqrt{2}}(e_p+e_q)}S^{2n-1}$ by showing that the
reconstruction formula applied to $e_r$ is not orthogonal to $e_s$.
We let $a=\frac{1}{\sqrt{2}}(e_p+e_q)$.
\begin{align*}\frac{\#
A}{(2n-1)}  \langle a, e_r\rangle  \langle a, e_s\rangle+\sum_{U\in
A}\langle U a,e_r\rangle  \langle U a, e_s\rangle&=0+ \sum_{U\in
A_{p,q,r,s,1}\cup A_{p,q,r,s,-1}}\langle U a,e_r\rangle\langle U a, e_s\rangle\\
&= \sum_{U\in A_{p,q,r,s,1}}\frac{1}{2}+\sum_{U\in
A_{p,q,r,s,-1}}-\frac{1}{2}\\
&= \frac{1}{2}(\# A_{p,q,r,s,1}-\# A_{p,q,r,s,-1})\neq0
\end{align*}
Thus, $\left\{\sqrt{\frac{\#
A}{(2n-1)}}\frac{1}{\sqrt{2}}(e_p+e_q)\right\}\cup\left\{U(\frac{1}{\sqrt{2}}(e_p+e_q))\right\}_{U\in
A}$ is not a tight frame for $\R^{2n}$ by Fact \ref{factBasis} and
hence $(f_U)_{U\in A}$ is not a moving FUNTF for $S^{2n-1}$.  Thus,
if $A\subset \mathbf{A}_{2n}$ is not balanced then $(f_U)_{U\in A}$
is not a moving FUNTF for $S^{2n-1}$.

Assume that $A\subset \mathbf{A}_{2n}$ is balanced.  We will show
that $\left\{\sqrt{\frac{\# A}{(2n-1)}}a\right\}\cup\{U(a)\}_{U\in
A}$ is a tight frame for $R^{2n}$ for all $a\in S^{2n-1}$.   Let
$a=(a_i)_{i=1}^{2n}\in S^{2n-1}\subset \R^{2n}$ and $1\leq p,q\leq
2n$ with $p\neq q$.

\begin{align*} \frac{\#
A}{(2n-1)}&\langle a, e_p\rangle\langle  a, e_q \rangle+\sum_{U\in
A}\langle U a,e_p\rangle \langle U a, e_q\rangle = \frac{\#
A}{(2n-1)}a_p a_q +
\sum_{U\in A}\vp_{k_p}\vp_{k_q}a_{k_p}a_{k_q}\\
&= \frac{\# A}{(2n-1)}a_p a_q +\sum_{U\in A_{p,q}}-a_{q}a_{p}
+\sum_{\substack{r\neq p,r\neq q \\ s\neq p,s\neq q}}\left(
\sum_{U\in A_{p,q,r,s,1}}a_{r}a_{s}+\sum_{U\in A_{p,q,r,s,-1}}-a_{r}a_{s}\right)\\
&= \frac{\# A}{(2n-1)}a_p a_q -\#A_{p,q}a_{q}a_{p}
+0=0\quad\quad\textrm{ as }A\textrm{ is balanced.}\\
\end{align*}
Thus the reconstruction formula applied to $e_p$ is orthogonal to
$e_q$.  We now let $a=(a_i)_{i=1}^{2n}\in S^{2n-1}\subset \R^{2n}$
and $1\leq p\leq 2n$.

\begin{align*} \frac{\#
A}{2n-1}\langle a, e_p\rangle^2+\sum_{U\in A}\langle U
a,e_p\rangle^2  &= \frac{\# A}{2n-1}a^2_p +
\sum_{U\in A}a_{k_p}^2\\
&= \frac{\# A}{2n-1}a^2_p +
\sum_{q\neq p}\sum_{U\in A_{p,q}}a_{q}^2\\
&= \frac{\# A}{2n-1}a^2_p + \sum_{q\neq p}\frac{\# A}{2n-1}a_{q}^2\quad\quad\textrm{ as $A$ is balanced.}\\
&= \frac{\# A}{2n-1}\quad\quad\textrm{ as }\|a\|=1
\end{align*}
Thus, $\left\{\sqrt{\frac{\#
A}{(2n-1)}}a\right\}\cup\left\{U(a)\right\}_{U\in A}$ is a tight
frame for $R^{2n}$ with frame bound $\frac{\# A}{2n-1}$ by Fact
\ref{factBasis}. This gives us that $\{U(a)\}_{U\in A}$ is a tight
frame for $T_a S^n$ for all $a\in S^n$ by Fact \ref{FactUnion}, and
hence $\{f_U\}_{U\in A}$ is a moving FUNTF for $S^{2n-1}$.
\end{proof}

By Lemma \ref{LemmaBalanced} and Theorem \ref{ThmMain}, we have for
all $n\in\N$, that $(f_U)_{U\in \mathbf{A}_{2n}}$ is a moving FUNTF
for $S^{2n-1}$.  Thus, as in the proof of Lemma \ref{LemmaBalanced},
$S^{2n-1}$ has a moving FUNTF of $\frac{(2n)!}{n!}$ vector fields
for all $n\in\N$.  The natural next step is to find FUNTFs comprised
of fewer vector fields.  That is, our goal is to now find a proper
subset of $\mathbf{A}_{2n}$ which is balanced.

\begin{thm}
For all $n\in\N$, there exists a balanced subset $A\subset
\mathbf{A}_{2n}$ with $\# A=(2n-1) 2^{n-1}$. In particular,
$S^{2n-1}$ has a moving FUNTF of $(2n-1)2^{n-1}$ vector fields for
all $n\in\N$.
\end{thm}
\begin{proof}
Let $n\in\N$.  Our first goal is to create a subset $B\subset
\{1,..,2n\}^{2n}$ such that
\begin{enumerate}
\item[i)] If $(k_i)_{i=1}^{2n}\in B$ then $k_i\neq i$ and $k_{{k_i}}=i$ for all $1\leq i\leq
2n$\\
\item[ii)] If $1\leq r,s\leq 2n$ and $r\neq s$ then there exists $(k_i)_{i=1}^{2n}\in B$
and $1\leq i\leq 2n$ such that $k_r=s$\\
\item[iii)] If $(k_i)_{i=1}^{2n},(\ell_i)_{i=1}^{2n}\in B$ then
$k_i\neq\ell_i$ for all $1\leq i\leq 2n$\\
\end{enumerate}
Essentially, $B$ would be a set of ways to pair up the components of
$\R^{2n}$ such that any two components are paired up by exactly one
permutation from $B$. Assuming we have created such a set $B$, we
may define a set $A$ by
$$A=\left\{U_{(\vp_i,k_i)_{i=1}^{2n}}\in\mathbf{A}_{2n}:\vp_{1}=1,\, \vp_i=-\vp_{k_i},\,(k_i)_{i=1}^{2n}\in B. \right\}$$
That is, we define $A$ by pairing up each permutation in $B$ with
all possible sequences $(\vp_i)_{i=1}^{2n}$ such that $\vp_1=1$. Due
to the restrictions on $B$, it is straightforward to check that the
resulting set $A$ will be balanced. The set $B$ contains $2n-1$
elements, and for each choice of $(k_i)_{i=1}^{2n}\in B$ there are
$2^{n-1}$ choices for $(\vp_i)_{i=1}^{2n}$.  Thus, $\# A=(2n-1)
2^{n-1}$.  All that remains is to show that such a set $B$ exists.

To create $B$, we will construct a $2n\times2n$ symmetric matrix $M$
such that every row (and column) of $M$ is a permutation of
$\{0,1,...,2n-1\}$ and the diagonal of $M$ is constant 0.  Given
such a matrix $M=[m_{i,j}]_{1\leq i,j\leq 2n}$, for all $1\leq
i,j<2n$ we let $k^j_i$ be the column number whose entry in the $i$th
row of $M$ equals $j$. Then we may set
$B:=\left\{(k^j_i)_{i=1}^{2n}\right\}_{1\leq j<2n}$. We have that
$i)$ is satisfied as $M$  has 0 diagonal and is symmetric. We have
that $ii)$ is satisfied as if $1\leq r,s\leq 2n$ with $r\neq s$ then
$k^{m_{r,s}}_{r}=s$.  We have that $iii)$ is satisfied as each row
of $M$ is a permutation of $\{0,1,...,2n-1\}$.

Thus, to create $B$, we need to construct a $2n\times2n$ symmetric
matrix $M$ such that every row (and column) of $M$ is a permutation
of $\{0,1,...,2n-1\}$ and the diagonal of $M$ is constant 0.  We
create $M=[m_{i,j}]_{1\leq i,j\leq 2n}$ by defining $m_i,j$ for each
$1\leq i,j\leq 2n$ by
\begin{center}
$m_{i,j} = \begin{cases}
0 & \text{if $i = j$},\\
(i+j-2) \text{ mod}(2n-1)+1 & \text{if $i\neq j$ and $1\leq i,j < 2n$},\\
(2i-2)\text{ mod}(2n-1)+1 & \text{if $j = 2n$ and $1 \leq i < n$},\\
(2j-2)\text{ mod}(2n-1)+1 & \text{if $i = 2n$ and $1 \leq j < n$}.\\
\end{cases}$
\end{center}

If we were to write out the matrix, it would look like:
\begin{center}
$\left( \begin{array}{c c c c c c c c c c c c c c}
0 & 2 & 3 & 4 & \cdots & &  &\cdots & 2n\!-\!4& 2n\!-\!3 & 2n\!-\!2 & 2n\!-\!1  & 1\\
2 & 0 & 4 & 5 & \cdots & &  &\cdots & 2n\!-\!3& 2n\!-\!2 & 2n\!-\!1 & 1 & 3 \\
3 & 4 & 0 & 6 & \cdots &  &  & \cdots & 2n\!-\!2& 2n\!-\!1 & 1 & 2 & 5 \\
 \vdots & & & \ddots & & & & &   &  & & &  \vdots \\
2n-3 & 2n\!-\!2 & 2n\!-\!1 & 1 & \cdots & && \cdots & 2n\!-\!7 & 0 & 2n\!-\!5 & 2n\!-\!4  & 2n\!-\!6\\
 2n\!-\!2 & 2n\!-\!1 & 1 & 2 & \cdots & && \cdots & 2n\!-\!6& 2n\!-\!5 & 0 & 2n\!-\!3   & 2n\!-\!4\\
2n\!-\!1 & 1 & 2 & 3 & \cdots & &  & \cdots&  2n\!-\!5 & 2n\!-\!4 & 2n\!-\!3 &   0 & 2n\!-\!2 \\
1 & 3 & 5 & 7 & \cdots & 2n\!-\!1 &2&\cdots &2n\!-\!8 & 2n\!-\!6&2n\!-\!4& 2n\!-\!2 & 0\\
\end{array} \right)$
\end{center}

If we switch the variables $i$ and $j$ in the definition of
$m_{i,j}$, we leave $m_{i,j}$ unchanged. Hence, $M$ is symmetric.
Setting $m_{i,j}=0$ if $i=j$ guarantees that $M$ has constant 0
diagonal.  To show that each row of $M$ is a permutation of
$\{0,1,...,2n-1\}$ we first consider $1\leq i< 2n$. The sequence
$\left((i+j-2) \text{ mod}(2n-1)\right)_{j=1}^{2n-1}$ is a
permutation of $(j)_{j=0}^{2n-2}$ and hence $\left((i+j-2) \text{
mod}(2n-1)+1\right)_{j=1}^{2n-1}$ is a permutation of
$(j)_{j=1}^{2n-1}$.  The $i$th row of $M$ is formed by concatenating
$0$ as the $(2n)$th element of the sequence $\left((i+j-2) \text{
mod}(2n-1)+1\right)_{j=1}^{2n-1}$ then  switching the $i$th and the
$(2n)$th element.  Thus the $i$th row of $M$ is a permutation of
$\{0,1,...,2n-1\}$ when $1\leq i<2n$.  For the case $i=2n$, we have
that $\left((2j-2)\text{ mod}(2n-1)\right)_{j=1}^{2n-1}$ is a
permutation of $(j)_{j=0}^{2n-2}$ as $2$ and $2n-1$ are relatively
prime, and hence $\left((2j-2)\text{
mod}(2n-1)+1\right)_{j=1}^{2n-1}$ is a permutation of
$(j)_{j=1}^{2n-1}$.  The $(2n)$th row of $M$ is formed by
concatenating $0$ as the $(2n)$th element of the sequence
$\left((2j-2)\text{ mod}(2n-1)+1\right)_{j=1}^{2n-1}$, and hence is
a permutation of $(j)_{j=0}^{2n-1}$.  Thus, we have formed a matrix
$M$ satisfying all our desired properties, and the proof is
complete.

\end{proof}

\end{document}